\documentclass[10pt]{amsart}

\usepackage[authoryear]{natbib}

\RequirePackage[OT1]{fontenc}
\RequirePackage{amsthm,amsmath,amssymb}
\usepackage[colorlinks=true, urlcolor=blue, linkcolor=black, citecolor=black, pdftex]{hyperref}
\usepackage{graphicx}
\usepackage[usenames]{color}

\usepackage{epstopdf} 

\def\var{{\rm var}}

\def\1{1\hskip-2.6pt{\rm l}}

\def\R{{\mathbb{R}}}

\def\E{{\mathbb{E}}}
\def\P{{\mathbb{P}}}

\def\A{{\mathcal A}}

\def\F{{\mathcal{F}}}

\def\M{\mathcal M}
\def\NN{{{\mathcal N}}}

\def\X{{\mathcal{X}}}
\def\S{{\mathcal S}_{n}}
\def\T{{\mathcal T}}

\def\eps{{\xi}}

\newcommand{\pen}{\mathop{\rm pen}\nolimits}
\newcommand{\eref}[1]{(\ref{#1})}
\newcommand{\pa}[1]{\left({#1}\right)}
\newcommand{\cro}[1]{\left[{#1}\right]}
\newcommand{\ab}[1]{\left|{#1}\right|}
\newcommand{\ac}[1]{\left\{{#1}\right\}}

\newtheorem{thm}{Theorem}
\newtheorem{lem}{Lemma}
\newtheorem{prop}{Proposition}
\newtheorem{cor}{Corollary}

\newtheorem{Ass}{Assumption}

\def\telque{\big |}

\def\<{{\langle}}
\def\>{{\rangle}}

\usepackage{memhfixc}  

\begin{document}
\title[Bernstein-type inequality]{A Bernstein-type inequality for suprema of random processes with applications to model selection in non-Gaussian regression}
\date{November 2008, revised july 2009} 
\author{Yannick Baraud} 
\address{Universit\'e de Nice Sophia-Antipolis, Laboratoire J-A Dieudonn\'e,
  Parc Valrose, 06108 Nice cedex 02} 
\email{baraud@unice.fr}
\begin{abstract}
Let $\pa{X_{t}}_{t\in T}$ be a family of real-valued centered random variables indexed by a countable set $T$. In the first part of this paper, we establish exponential bounds for the deviation probabilities of the supremum $Z=\sup_{t\in T}X_{t}$ by using the generic chaining device introduced in Talagrand~\citeyearpar{MR1361756}. Compared to concentration-type inequalities, these bounds offer the advantage to hold under weaker conditions on the family $\pa{X_{t}}_{t\in T}$. The second part of the paper is oriented towards statistics. We consider the regression setting $Y=f+\eps$ where $f$ is an unknown vector of $\R^{n}$ and $\eps$ is a random vector the components of which are independent, centered and admit finite Laplace transforms in a neighborhood of 0. Our aim is to estimate $f$ from the observation of $Y$ by mean of a model selection approach among a collection of linear subspaces of $\R^{n}$. The selection procedure we propose is based on the minimization of a penalized criterion 
the penalty of which is calibrated by using the deviation bounds established in the first part of this paper. More precisely, we study suprema of random variables of the form $X_{t}=\sum_{i=1}^{n}t_{i}\eps_{i}$ when $t$ varies among the unit ball of a linear subspace of $\R^{n}$. 
We finally show that our estimator satisfies some oracle-type inequality under suitable assumptions on the metric structures of the linear spaces of the collection.
\end{abstract}
\maketitle
%

\section{introduction}
\subsection{What is this paper about?}
The present paper contains two parts. The first one is oriented towards probability. We consider  a family $\pa{X_{t}}_{t\in T}$ of real-valued centered random variables indexed by a countable set $T$ and give an exponential bound for the probability of deviation of the supremum $Z=\sup_{t\in T}X_{t}$. 
The result is established under the assumption that the Laplace transforms of the increments $X_{t}-X_{s}$ for $s,t\in T$ satisfy some Bernstein-type bounds. This assumption is convenient to handle simultaneously the cases of subgaussian increments (which is the typical case in the literature) as well as more ``heavy tailed" ones for which the Laplace transform of $\pa{X_{s}-X_{t}}^{2}$ may be infinite in a neighborhood of 0. Under additional assumptions on the $X_{t}$, our result allows to recover (with worse constants) some deviation bounds based on concentration-type inequalities of $Z$ around its expectation. However our general result cannot be deduced from those inequalities. As we shall see, concentration-type inequalities could be false under the kind of assumptions we consider on the family $\pa{X_{t}}_{t\in T}$. 

The second part is oriented towards statistics. We consider the regression framework 
\begin{equation}\label{modreg0}
Y_{i}=f_{i}+\eps_{i},\ i=1,\ldots,n
\end{equation}
where $f=(f_{1},\ldots,f_{n})$ is an unknown vector of $\R^{n}$ and $\eps=(\eps_{1},\ldots,\eps_{n})$ is a random vector the components of which are independent, centered and admit suitable exponential moments. Our aim is to estimate $f$ from the observation of $Y=(Y_{1},\ldots,Y_{n})$ by mean of a model selection approach. More precisely, we start with a collection $\mathcal{S}=\ac{S_{m},\ m\in\M}$ of finite dimensional linear spaces $S_{m}$ to each of which we associate the least-squares estimator $\hat f_{m}\in  S_{m}$ of $f$. From the same data $Y$, our aim is to select some suitable estimator $\tilde f=\hat f_{\hat m}$ among the collection $\F=\ac{\hat f_{m},\ m\in\M}$ in such a way that the (squared) Euclidean risk of $\tilde f$ is as close as possible to the infimum of the risks over $\F$. The selection procedure we propose is based on the minimization of a penalized criterion 
the penalty of which is calibrated by using the deviation bounds established in the first part of this paper. More precisely, the penalty is obtained by studying the deviations of $\chi^{2}$-type random variables, that is, random variables of the form $\ab{\Pi_{S}\eps}_{2}^{2}$ where $\ab{\ }_{2}$ denotes the Euclidean norm and $\Pi_{S}$ the orthogonal projector onto a linear subspace $S$ of $\R^{n}$. To our knowledge, these deviation bounds in probability are new. We finally show that $\tilde f$ satisfies some oracle-type inequality under suitable assumptions on the metric structures of the $S_{m}$.

In the following sections, we situate the results of the present paper within the literature. 

\subsection{Controlling suprema of random processes}
Among the most common deviation inequalities, let us recall 
   
\begin{thm}[Bernstein's inequality]
Let $X_{1},\ldots,X_{n}$ be independent random variables and set $X=\sum_{i=1}^{n }\pa{X_{i}-\E(X_{i})}$. Assume that there exist nonnegative numbers  $v,c$ such that for all $k\ge 3$
\begin{equation}\label{bern-moment}
\sum_{i=1}^{n}\E\cro{\ab{X_{i}}^{k}}\le {k!\over 2}v^{2}c^{k-2},
\end{equation}
then for all $u\ge 0$
\begin{equation}\label{bern}
\P\pa{X \ge \sqrt{2v^{2}u}+cu}\le e^{-u}.
\end{equation}
Besides, for all $x\ge 0$, 
\begin{equation}\label{bern01}
\P\pa{X\ge x}\le \exp\pa{-{x^{2}\over 2(v^{2}+cx)}}.
\end{equation}
\end{thm}
In the literature, \eref{bern-moment} together with the fact that the $X_{i}$ are independent is sometime replaced by the weaker condition 
\begin{equation}\label{exp}
\E\pa{e^{\lambda X}}\le \exp\cro{{\lambda^{2}v^{2}\over 2(1-\lambda c)}},\ \ \ \ \forall \lambda\in (0,1/c)
\end{equation}
with the convention $1/0=+\infty$. Bernstein's inequality allows to derive deviation inequalities for a large class of distributions among which the Poisson, Laplace, Gamma or the Gaussian distributions (once suitably centered). In this latter case,~\eref{exp} holds with $c=0$.  Another situation  of interest is the case where the $X_{i}$ are i.i.d. with values in $[-c,c]$. Then~\eref{bern-moment} and~\eref{exp} hold with $v^{2}=\var(X_{1})$. 

In the recent years, many efforts have been done to extend these bounds to the deviations of suprema $Z$ of random variables $X_{t}$. When $T$ is a (countable) bounded subset of a metric space $(\X,d)$, a common technique is to use a {\it chaining device}. This approach seems to go back to Kolmogorov and was very popular in statistics in the 90s to control suprema of empirical processes with regard to the entropy of $T$, see van de Geer~\citeyearpar{Geer90} for example. However, this approach leads to pessimistic numerical constants that are in general too large to be used in statistical procedures. An alternative to chaining is the use of concentration inequalities. For example, when the $X_{t}$ are Gaussian, for all $u\ge 0$ we have 
\begin{equation}\label{gaussien}
\P\pa{Z\ge \E\pa{Z}+\sqrt{2v^{2}u}}\le e^{-u}\ \ \ {\rm where}\ \ \ v^{2}=\sup_{t\in T}{\rm var}(X_{t}).
\end{equation}
This inequality is due to Sudakov \& Cirel'son~\citeyearpar{MR0365680}. Compared to chaining,~\eref{gaussien} provides a powerful tool for controlling suprema of Gaussian processes as soon as one is able to  evaluate $\E(Z)$ sharply enough.  

It is the merit of  Talagrand~\citeyearpar{MR1361756} to extend this approach for the purpose of controlling suprema of bounded empirical processes, that is, for 
$X_{t}$ of the form $X_{t}=\sum_{i=1}^{n}t(\xi_{i})-\E\pa{t(\xi_{i})}$ where $\eps_{1},\ldots,\eps_{n}$ are independent random variables  and $T$ a set of uniformly bounded functions, say with values in $[-c,c]$. From Talagrand's inequality, one can deduce deviation bounds with respect to $\E(Z)$ of the form 
\begin{equation}\label{bern0}
\P\cro{Z\ge C\pa{\E(Z)+ \sqrt{v^{2}u}+cu}}\le \exp\pa{-u}\ \ \mbox{for all}\ u\ge 0
\end{equation}
where $v^{2}=\sup_{t\in T}{\rm var}\pa{X_{t}}$ and $C$ is a positive numerical constant. Apart from the constants,~\eref{bern0} and~\eref{bern} have a similar flavor even though the boundness assumption on the elements of $T$ seems too strong compared to conditions~\eref{bern-moment} or~\eref{exp}. 

As the original result by Talagrand involved suboptimal numerical constants, many efforts were made to recover it with sharper ones. A first step in this direction is due to Ledoux~\citeyearpar{MR1399224} by mean of nice entropy and tensorisation arguments. Then, further refinements were made on Ledoux's result by Massart~\citeyearpar{MR1782276}, Rio~\citeyearpar{MR1955352} and Bousquet~\citeyearpar{MR1890640}, the latter author achieving the best possible result in terms of constants. For a nice introduction to these inequalities (and their applications to statistics) we refer the reader to the book by Massart~\citeyearpar{MR2319879}. Other improvements upon~\eref{bern0} have been done in the recent years. In particular Klein \& Rio~\citeyearpar{MR2135312} generalized the result to the case
\begin{equation}\label{defXti}
X_{t}=\sum_{i=1}^{n}\overline{X}_{i,t}
\end{equation}
where for each $t\in T$, $\pa{\overline{X}_{i,t}}_{i=1,\ldots,n}$ are independent (but not necessarily i.i.d.) centered random with values in $[-c,c]$. 

In the present paper, the result we establish holds under different assumptions than the ones leading to inequalities such as~\eref{bern0}. First, as pointed out by Jonas Kahn, an inequality such as~\eref{bern0} could be false under the kind of assumptions we consider on the family $\pa{X_{t}}_{t\in T}$. In the counter-example we give in Section~\ref{sect:I} (it is a slight modification of the one Jonas Kahn gave to us), we see that $Z$ may deviate from $\E(Z)$ on a set the probability of which may not be exponentially small. Moreover,  even in the more common situation where $X_{t}$ is of the form~\eref{defXti}, we establish deviation inequalities that are available for possibly unbounded random variables $\overline{X}_{i,t}$ which is beyond the scope of the concentration inequalities proven in Bousquet~\citeyearpar{MR1890640} and Klein \& Rio~\citeyearpar{MR2135312}. 

Even though it was originally introduced to bound $\E(Z)$ from above, {\it generic chaining} as described in Talagrand's book~\citeyearpar{MR2133757} provides another way of establishing deviation bounds for $Z$. Talagrand's approach relies on the idea of decomposing $T$ into partitions rather than into nets as it was usually done before with the classical chaining device. Denoting by $e_{1},\ldots,e_{k}$ the canonical basis of $\R^{k}$ and $\eps^{(1)},\ldots,\eps^{(k)}$ i.i.d. random vectors of $\R^{n}$ with common distribution $\mu$, generic chaining was used in Mendelson {\it et al}~\citeyearpar{MR2373017}
and Mendelson~\citeyearpar{MR2368981} to study the properties of the random operator $\Gamma: t\mapsto k^{-1/2}\sum_{i=1}^{k}\<\eps^{(i)},t\>e_{i}$ defined for $t$ in the unit sphere $T$ of $\R^{n}$ (which we endow with its usual scalar product $\<.,.\>$). Their results rely on the control of suprema of random variables of the form $X_{t}=k^{-1}\sum_{i=1}^{k}\<\eps^{(i)},t\>$ for $t\in T$. When $k=1$, this form of $X_{t}$ is analogous to that we consider in our statistical application. However, the deviation bounds obtained in Mendelson {\it et al}~\citeyearpar{MR2373017} 
and Mendelson~\citeyearpar{MR2368981} require that $\mu$ be subgaussian which we do not want to assume here. Closer to our result is Theorem~3.3 in Klartag \& Mendelson~\citeyearpar{MR2149924} which bounds on a set of probability at least $1-\delta$ (for some $\delta\in (0,1)$) the supremum $Z=\sup_{t\in T}\ab{X_{t}}$. Unfortunately, their bound involves non-explicit constants (that depend on $\delta$) which makes it useless for statistical issues.  

Our approach also uses generic chaining. With such a technique, the inequalities we get suffer from the usual drawback that the numerical constants are non-optimal but at least allow a suitable control of the $\chi^{2}$-type random variables we consider in the statistical part of this paper. To our knowledge, these inequalities are new. 

\subsection{From the control of $\chi^{2}$-type random variables to model selection in regression} 
The reason why $\chi^{2}$-type random variables naturally emerge in the regression setting is the following one. Let $S$ be a linear subspace of $\R^{n}$. The classical least-squares estimator of $f$ in $S$ is given by $\hat f=\Pi_{S}Y=\Pi_{S}f+\Pi_{S}\xi$ and since the Euclidean (squared) distance beween $f$ and $\hat f$ decomposes as 
\[
\ab{f-\hat f}_{2}^{2}=\ab{f-\Pi_{S}f}_{2}^{2}+\ab{\Pi_{S}\xi}_{2}^{2}
\]
the study of the quadratic loss $\ab{f-\hat f}_{2}^{2}$ requires that of its random component $\ab{\Pi_{S}\xi}_{2}^{2}$. This quantity is called a $\chi^{2}$-type random variable by analogy to the Gaussian case. Its study is connected to that of suprema of random variables by the formula
\begin{equation}\label{x2}
\ab{\Pi_{S}\xi}_{2}=\sup_{t\in T}X_{t}=Z\ \ \mbox{with}\ \ X_{t}=\sum_{i=1}^{n}\xi_{i}t_{i}
\end{equation}
where $T$ is the unit ball of $S$ (or a countable and dense subset of it). The control of such random variables is at the heart of the model selection scheme. When $\eps$ is a standard Gaussian vector of $\R^{n}$,  Birg\'e \& Massart~\citeyearpar{MR1848946} used~\eref{gaussien} to control the probability of deviation of $\ab{\Pi_{S}\xi}_{2}$ with respect to its expectation. The strong integrability properties of the $\xi_{i}$ allows to handle very general collections of models. By using chaining techniques, these results were extended to the subgaussian case (that is for $\pm\xi_{i}$ satisfying~\eref{exp} with $c=0$ for all $i$) in Baraud, Comte \& Viennet~\citeyearpar{MR1845321}. Similarly, very few assumptions were required on the collection to perform model selection. Baraud~\citeyearpar{MR1777129} considered the case where the $\xi_{i}$ only admit few finite moments. There, the weak integrability properties of the $\eps_{i}$ induced severe restrictions on the collection of models ${\mathcal S}$. Typically, for all $D\in\ac{1,\ldots, n}$ the number of models $S_{m}$ of a given dimension $D$ had to be at most polynomial with respect to $D$, the degree of the polynomial depending on the number of finite moments of $\eps_{1}$.  

To our knowledge, the intermediate case where the random variables $\pm\eps_{i}$ admit exponential moments of the form~\eref{exp} for all $i$ (with $c\neq 0$ to exclude the already known subgaussian case) has remained open for general collections of models. In this context, the concentration-type inequality obtained in Klein \& Rio~\citeyearpar{MR2135312} cannot be used to control $\ab{\Pi_{S}\xi}_{2}$ as it would require that the $\eps_{i}$ be bounded. An attempt at relaxing this boundedness assumption on the $\eps_{i}$ can be found in Bousquet~\citeyearpar{MR2073435}. There, the author considered the situation where $T$ is a subset of $[-1,1]^{n}$ and the $\xi_{i}$ independent and centered random variables satisfying
\begin{equation}\label{moment}
\E\cro{\ab{\xi_{i}}^{k}}\le {k!\over 2}\sigma^{2}c^{k-2},\ \ \forall \ k\ge 2.
\end{equation}
Note that~\eref{moment} implies~\eref{exp} with $v^{2}=v^{2}(t)=\ab{t}_{2}^{2}\sigma^{2}$. The result by Bousquet provides an analogue of~\eref{bern0} with $v^{2}$ replaced by $n\sigma^{2}$ although one would expect the smaller (and usual) quantity $v^{2}=\sup_{t\in T}v^{2}(t)$. Because of this, the resulting inequality turns out to be useless at least for the statistical application we have in mind. This fact has already been pointed out by Marie Sauv\'e in Sauv\'e~\citeyearpar{MSauve}. Sauv\'e also tackled the problem of model selection when the $\eps_{i}$ satisfy~\eref{moment}. Compared to Baraud~\citeyearpar{MR1777129}, her condition on the collection of models is weaker in the sense that the number of models with a given dimension $D$ is allowed to be exponentially large with respect to $D$. However, the collection she considered only consists of linear spaces $S_{m}$ with a specific form (leading to regressogram estimators). Besides, her selection procedure was relying on a known upper bound on $\max_{i=1,\ldots,n}\ab{f_{i}}$ which can be unrealistic in practice. Unlike Marie Sauv\'e's, our procedure does not depend on such an upper bound and allows for more general linear spaces $S_{m}$.    

\subsection{Organisation of the paper and main notations}
The paper is organized as follows. We present our deviation bound for $Z$ in Section~\ref{sect:I}. The statistical application is developed in Sections~\ref{sect:stats} and~\ref{sec:unifions}. In Section~\ref{sect:stats} we consider particular cases of collections  $\mathcal{S}$ of interest, the general case being considered in Section~\ref{sec:unifions}. Section~\ref{Proof} is devoted to the proofs. 

Along the paper we assume that $n\ge 2$ and use the following notations. We denote by $e_{1},\ldots,e_{n}$ the canonical basis of $\R^{n}$ which we endow with the Euclidean inner product denoted $\<.,.\>$. For $x\in\R^{n}$, we set $|x|_{2}=\sqrt{\<x,x\>}$, $|x|_{1}=\sum_{i=1}^{n}|x_{i}|$ and $|x|_{\infty}=\max_{i=1,\ldots,n}|x_{i}|$. The linear span of a family $u_{1},\ldots,u_{k}$ of vectors is denoted by ${\rm Span\!}\ac{u_{1},\ldots,u_{k}}$. The quantity $|I|$ is the cardinality of a finite set $I$.  Finally, $\kappa$ denotes the numerical constant $18$. It appears first in the control of the deviation of $Z$ when applying Talagrand's chaining argument and then all along the paper. It seemed interesting to stress up the influence of this constant in the model selection procedure we propose. 

\section{A Talagrand-type Chaining argument for controlling suprema of random variables}\label{sect:I}

Let $\pa{X_{t}}_{t\in T}$ be a family of real valued and centered random variables indexed by a countable and nonempty set $T$.  Fix some $t_{0}$ in $T$ and set
\[
Z=\sup_{t\in T}\pa{X_{t}-X_{t_{0}}}\ \ \ {\rm and}\ \ \ \overline{Z}=\sup_{t\in T}\ab{X_{t}-X_{t_{0}}}.
\]
Our aim is to give a probabilistic control of the deviations of $Z$ (and $\overline Z$). We make the following assumptions
\begin{Ass}\label{momexp}
There exist two distances $d$ and $\delta$ on $T$ and a nonnegative constant $c$ such that for all $s,t\in T$ ($s\ne t$) 
\begin{equation}\label{debase}
\E\cro{e^{\lambda(X_{t}-X_{s})}}\le \exp\cro{{\lambda^{2}d^{2}(s,t)\over 2(1-\lambda c\delta(s,t))}},\ \ \forall \lambda\in\left[0, {1\over c\delta(s,t)}\right)
\end{equation}
with the convention $1/0=+\infty$.
\end{Ass}
Note that $c=0$ corresponds to the particular situation where the increments of the process $X_{t}$ are {\it subgaussian}. 

Besides Assumption~\ref{momexp}, we also assume in this section that $d$ and $\delta$ derive from norms. This is the only case we need to consider to handle the statistical problem described in Section~\ref{sect:stats}. Nevertheless, a more general result with arbitrary distances can be found in Section~\ref{Proof}. 

\begin{Ass}\label{lineaire}
Let $S$ be a linear space with finite dimension $D$ endowed with two arbitrary norms denoted $\|\ \|_{2}$ and $\|\ \|_{\infty}$ respectively. Define for $s,t\in S$, $d(s,t)=\|t-s \|_{2}$ and $\delta(s,t)=\|s-t \|_{\infty}$ and assume that for constants $v>0$ and $c\ge 0$,
\[
T\subset \ac{t\in S\ \telque\ \|t-t_{0}\|_{2}\le v,\ \ c\|t-t_{0}\|_{\infty}\le b}.
\]
\end{Ass}

Then, the following result holds. 
\begin{thm}\label{norm}
Under Assumptions~\ref{momexp} and~\ref{lineaire}, 
\begin{equation}\label{svanorm}
\P\cro{Z\ge \kappa\pa{\sqrt{v^{2}(D+x)}+b(D+x)}}\le e^{-x},\ \ \forall x\ge 0
\end{equation}
with $\kappa=18$. Moreover 
\begin{equation}\label{vanorm}
\P\cro{\overline{Z}\ge \kappa\pa{\sqrt{v^{2}(D+x)}+b(D+x)}}\le 2e^{-x},\ \ \forall x\ge 0.
\end{equation}
\end{thm}
Since $S$ is separable, the result easily extends to the case where $T\subset S$ is not countable provided the paths $t\mapsto X_{t}$ are continuous with probability 1 (with respect to $\|\ \|_{2}$ or $\|\ \|_{\infty}$, both norms being equivalent on $S$). 

\subsection{Connections with deviations inequalities with respect to $\E(Z)$}
In this section we make some connections between our bound~\eref{svanorm} and inequalities~\eref{gaussien} and~\eref{bern0}. Along this section, $T$ is the unit ball of the linear span $S$ of an orthonormal system $\ac{u_{1},\ldots,u_{D}}$. Both norms $\ab{\ }_{2}$ and $\ab{\ }_{\infty}$ being equivalent on $S$, we set 
\[
\Lambda_{2}(S)=\sup_{t\in T\setminus\ac{0}}{\ab{t}_{\infty}\over \ab{t}_{2}}<+\infty.
\]
Note that $\Lambda_{2}(S)$ depends on the metric structure of $S$. In all cases, $\Lambda_{2}(S)\le 1$, this bound being achieved for $S={\rm Span}\ac{e_{1},\dots,e_{D}}$ for example. However, $\Lambda_{2}(S)$ can be much smaller, equal to $\sqrt{D/n}$ for example, when $n=kD$ for some positive integer $k$ and 
$u_{j}=\pa{e_{(j-1)k+1},\ldots,e_{jk}}/\sqrt{k}$ for $j=1,\ldots,D$. The set $T$ fulfills Assumption~\ref{lineaire} with $t_{0}=0$, $d(s,t)=\ab{s-t}_{2}$, $\delta(s,t)=\ab{s-t}_{\infty}$, $v=1$ and $b=c\Lambda_{2}(S)$. 
Let $\eps=(\eps_{1},\ldots,\eps_{n})$ be a random vector of $\R^{n}$ with i.i.d. components of common variance 1. We consider the process defined on $T$ by $X_{t}=\<t,\eps\>$ and note that in this case $Z=\sup_{t\in T}X_{t}=\ab{\Pi_{S}\eps}_{2}$. Besides, by using Jensen's inequality
\begin{equation}\label{esp}
\E\cro{Z}=\E\cro{\sqrt{\sum_{j=1}^{D}\<u_{j},\eps\>^{2}}}\le \sqrt{D}.
\end{equation} 

{\it The Gaussian case:}
Assume that the $\xi_{i}$ are standard Gaussian random variables. On the one hand, since $\sup_{t\in T}\var(X_{t})=1$ we deduce from Sudakov \& Cirel'son's bound~\eref{gaussien} together with~\eref{esp} 
\begin{equation}\label{cs1}
\P\pa{Z\ge \sqrt{D}+\sqrt{2x}}\le e^{-x},\ \forall x\ge 0.
\end{equation}
On the other hand, since~\eref{exp} holds with $c=0$, for all $s,t\in S$ and $\lambda\ge 0$
\begin{eqnarray*}
\E\cro{e^{\lambda(X_{t}-X_{s})}}&=&\prod_{i=1}^{n}\E\cro{e^{\lambda\eps_{i}\pa{t_{i}-s_{i}}}}\le \prod_{i=1}^{n}\exp\cro{{\lambda^{2}\ab{t_{i}-s_{i}}^{2}\over 2}}\\
&\le& \exp\cro{{\lambda^{2}\ab{t-s}_{2}^{2}\over 2}}.
\end{eqnarray*}
Consequently,~\eref{debase} holds with $c=0$ and one can apply Theorem~\ref{norm} to get
\begin{equation}\label{cs2}
\P\cro{Z\ge \kappa\pa{\sqrt{D}+\sqrt{x}}}\le \P\pa{Z\ge \kappa\sqrt{D+x}}\le e^{-x},\ \forall x\ge 0.
\end{equation}
Apart from the numerical constants, it turns out that~\eref{cs1} and~\eref{cs2} are similar in this case.
 
\vspace{2mm}
{\it The bounded case: }
Let us assume that the $\eps_{i}$ take their values in $[-a,a]$ for some $a\ge 1$. We can apply the bound given by Klein \& Rio~\citeyearpar{MR2135312} with $v=1$ and $c=a\Lambda_{2}(S)$ in~\eref{bern0} which together with~\eref{esp} gives for a suitable constant $C>0$, 
\begin{equation}\label{R1}
\P\cro{Z\ge C\pa{\sqrt{D}+\sqrt{x}+a\Lambda_{2}(S)x}}\le \exp\pa{-x}\ \ \mbox{for all}\ x\ge 0.
\end{equation}

When the $\xi_{i}$ are bounded, there are actually two ways of applying Theorem~\ref{norm}. One relies on the fact that the random variables $\pm \xi_{i}$ satisfy~\eref{exp} with $v=1$ and $c=a$ for all $i$. Hence, whatever $s,t\in S$ and $\lambda\le (a\ab{s-t}_{\infty})^{-1}$,
\begin{eqnarray*}
\E\cro{e^{\lambda(X_{t}-X_{s})}}&=&\prod_{i=1}^{n}\E\cro{e^{\lambda\eps_{i}\pa{t_{i}-s_{i}}}}\le \prod_{i=1}^{n}\exp\cro{{\lambda^{2}\ab{t_{i}-s_{i}}^{2}\over 2(1-\lambda a\ab{t-s}_{\infty})}}\\
&\le& \exp\cro{{\lambda^{2}\ab{t-s}_{2}^{2}\over 2\pa{1-\lambda a \ab{t-s}_{\infty}}}}
\end{eqnarray*}
and since Assumption~\ref{momexp} holds with $c=a$ and we get from Theorem~\ref{norm} 
\begin{equation}\label{B1}
\P\cro{Z\geq \kappa\pa{\sqrt{D}+\sqrt{x}+a\Lambda_{2}(S)x+a\Lambda_{2}(S)D}}\le e^{-x},\ \forall x\ge 0. 
\end{equation}
Inequalities~\eref{R1} and~\eref{B1} essentially differ by the fact that the latter involves the extra term $a\Lambda_{2}(S)D$. Hence, we recover~\eref{R1} only for those $S$ bearing some specific metric structure for which $\Lambda_{2}(S)\le C'(a\sqrt{D})^{-1}$ for some numerical constant $C'>0$. 

The other way of using Theorem~\ref{norm} is to note that the random variables $\pm\xi_{i}$ are subgaussian (because they are bounded) and therefore satisfy~\eref{exp} with $v=a$ and $c=0$. By arguing as in the Gaussian case, Assumption~\ref{momexp} holds with $d(s,t)=a\ab{s-t}_{2}$ for all $s,t\in S$, $c=0$ and Assumption~\ref{lineaire} is fulfilled with $v=a$ and $b=0$. We deduce from Theorem~\ref{norm}
\begin{equation}\label{B2}
\P\cro{Z\ge \kappa\pa{a\sqrt{D}+a\sqrt{x}}}\le e^{-x}\ \ \forall x\ge 0.
\end{equation}
Note that whenever $a$ is not too large compared to 1, this bound improves~\eref{R1} by avoiding the linear term $a\Lambda_{2}(S)x$. 

\subsection{ A counter-example}\label{jonas}
In this section we show that for the supremum $Z$ of a random process $\mathbf{X}=\pa{X_{t}}_{t\in T}$ satisfying~\eref{debase}  may not concentrate around $\E(Z)$. 
More precisely, let us show that~\eref{bern0} could be false under~~\eref{debase}. A simple counter-example is the following one. For $D\ge 1$, let $S={\rm Span}\ac{e_{1},\ldots,e_{D}}$, $T$ be the unit ball of $S$ and $\mathbf{X'}=\pa{X'_{t}}_{t\in T}$ the Gaussian process defined for $t\in T$ by $t\mapsto \<t, \eps\>$ where $\eps$ is a standard Gaussian vector of $\R^{n}$. For $p\in (0,1)$, define $\mathbf{X}$ as either $\mathbf{X'}$ with probability $p$ or the process $\mathbf{X''}$ identically equal to 0 with probability $1-p$. On the one hand, note that both processes $\mathbf{X'}$ and $\mathbf{X''}$ satisfy~\eref{debase} with $c=0$, $d(s,t)=\ab{s-t}_{2}$ for all $s,t\in S$ and therefore so does $\mathbf{X}$ (whatever $p$). On the other hand, since
\[
\E(Z)=p\E\cro{\sup_{t\in T}X_{t}'}=p\E\cro{\sqrt{\sum_{i=1}^{D}\eps_{i}^{2}}}\le p\sqrt{D}
\]
and $\sup_{t\in T}\var(X_{t})\le 1$,~\eref{bern0} would imply that for some positive numerical constant $C$ (that we can take larger than 1 with no loss of generality) whatever $p\in(0,1)$ and $u\ge 0$, 
\begin{eqnarray*}
\P\cro{Z\ge Cp\sqrt{D}+C\pa{\sqrt{u}+u}}&=&p\P\cro{\sqrt{\sum_{i=1}^{D}\eps_{i}^{2}}\ge Cp\sqrt{D}+C\pa{\sqrt{u}+u}}\\
&\le& e^{-u}.
\end{eqnarray*}
In particular, by taking $p=(2C)^{-1}\in (0,1)$ and $u=\log(2/p)$, we would get
\[
\P\cro{\sqrt{{1\over D}\sum_{i=1}^{D}\eps_{i}^{2}}\ge {1\over 2}+{C\over \sqrt{D}}\pa{\sqrt{\log(2/p)}+\log(2/p)}}\le {1\over 2}
\]
which is of course false by the law of large numbers for large values of $D$. 

\section{Applications to model selection in regression}\label{sect:stats}
Consider the regression framework given by~\eref{modreg0} and assume that for some known nonnegative numbers $\sigma$ and $c$ 
\begin{equation}\label{bernstein}
\log\E\cro{e^{\lambda \eps_{i}}}\le {\lambda^{2}\sigma^{2}\over 2(1-|\lambda| c)}\ \ \mbox{for all}\ \ \lambda\in(-1/c,1/c)\ \mbox{and}\ i=1,\ldots,n.
\end{equation}
Inequality~\eref{bernstein} holds for a large class of distributions (once suitably centered) including Gaussian, Poisson, Laplace or Gamma (among others). Besides,~\eref{bernstein} is fulfilled when the $\xi_{i}$ satisfy~\eref{moment} and therefore whenever these are bounded. 

Our estimation strategy is based on model selection. We start with a (possibly large) collection $\ac{S_{m},\ m\in\M}$ of linear subspaces (models) of $\R^{n}$ and associate to each of these the least-squares estimators $\hat f_{m}=\Pi_{S_{m}}Y$. Given a penalty function $\pen$ from $\M$ to $\R_{+}$, we define the penalized criterion ${\rm crit}(.)$ on  $\M$ by
\begin{equation}\label{crit}
{\rm crit}(m)=\ab{Y-\hat f_{m}}_{2}^{2}+\pen(m).
\end{equation}
In this section, we propose to establish risk bounds for the estimator of $f$ given by $\hat f_{\hat m}$ where the index $\hat m$ is selected from the data among $\M$ as any minimizer of  ${\rm crit}(.)$. 

In the sequel, the penalty $\pen$ will be based on some {\it a priori} choice of nonnegative numbers $\ac{\Delta_{m},\ m\in\M}$ for which we set 
\[
\Sigma=\sum_{m\in\M} e^{-\Delta_{m}}<+\infty.
\]
When $\Sigma=1$, the choice of the $\Delta_{m}$ can be viewed as that of a prior distribution on the models $S_{m}$. For related conditions and their interpretation, see Barron and Cover~\citeyearpar{MR1111806} or  Barron {\it et al}~\citeyearpar{MR1679028}.

In the following sections, we present some applications of our main result  (to be presented in Subsection~\ref{sect:main}) for some collections of linear spaces $\ac{S_{m},\ m\in\M}$ of interest.

\subsection{Selecting among histogram-type estimators}\label{sect100}
For a partition $m$ of $\ac{1,\ldots,n}$, $S_{m}$ denotes  the linear span of vectors of $\R^{n}$ the coordinates of which are constants on each element $I$ of $m$. In the sequel, we shall restrict to partitions $m$ the elements of which consist of consecutive integers. 

Consider a partition  $\mathfrak{m}$ of $\ac{1,\ldots,n}$ and $\M$ a collection of partitions $m$ such that $S_{m}\subset S_{\mathfrak{m}}$. We obtain the following result.

\begin{prop}\label{histo}
Let $a,b>0$. Assume that
\begin{equation}\label{condhisto}
|I|\ge a^{2}\log^{2}n,\ \ \forall I\in \mathfrak{m}.
\end{equation}
If for some $K>1$, 
\begin{equation}\label{p1}
\pen(m)\ge K\kappa^{2}\pa{\sigma^{2}+2c{(\sigma+c)(b+2)\over a\kappa}}\pa{|m|+\Delta_{m}},\ \ \forall m\in\M.
\end{equation}
the estimator $\hat f_{\hat m}$ satisfies 
\begin{equation}\label{inehisto}
\E\pa{\ab{f-\hat f_{\hat m}}_{2}^{2}}\le C(K)\cro{\inf_{m\in\M}\cro{\E\pa{\ab{f-\hat f_{m}}_{2}^{2}}+\pen(m)}+R}
\end{equation}
where  $C(K)$ is given by~\eref{CK} and
\begin{equation*}\label{R(c,L)}
R=\kappa^{2}\pa{\sigma^{2}+2c{(c+\sigma)(b+2)\over a\kappa}}\Sigma+2{(c+\sigma)^{2}(b+2)^{2}\over a^{2}n^{b}}.
\end{equation*}
\end{prop}

Note that when $c=0$, inequality~\eref{p1} holds as soon as 
\begin{equation}\label{penideal}
\pen(m)= K\kappa^{2}\sigma^{2}\pa{|m|+\Delta_{m}},\ \ \forall m\in\M.
\end{equation}
Besides, by taking $a=(\log n)^{-1}$ we see that condition~\eref{condhisto} becomes automatically satisfied and by letting $b$ tend to $+\infty$, inequality~\eref{inehisto} holds with $\pen$ given by~\eref{penideal} and $R=\kappa^{2}\sigma^{2}\Sigma$.

The problem of selecting among histogram-type estimators in this regression setting has recently been investigated in Sauv\'e~\citeyearpar{MSauve}. Her selection procedure is similar to ours with a different choice of the penalty term. Unlike hers, our penalty does not involve any known upper bound on $\ab{f}_{\infty}$. 

\subsection{Families of piecewise polynomials}\label{sect101}
In this section, we assume that $f=(F(x_{1}),\ldots, F(x_{n}))$ where $x_{i}=i/n$ for $i=1,\ldots,n$ and $F$ is an unknown function on $(0,1]$. Our aim is to estimate $F$ by a piecewise polynomial of degree not larger than $d$ based on a data-driven choice of a partition of $(0,1]$. 

In the sequel, we shall consider partitions $m$ of $\ac{1,\ldots,n}$ such that each element  $I\in m$ consists of at least $d+1$ consecutive integers. For such a partition, $S_{m}$ denotes the linear span of vectors of the form $(P(1/n),\ldots,P(n/n))$ where $P$ varies among the space of piecewise polynomials with degree not larger than $d$ based on the partition of $(0,1]$ given by
\[
\ac{\left({\min I-1\over n}, {\max I\over n}\right],\ I\in m}.
\]
Consider a partition  $\mathfrak{m}$ of $\ac{1,\ldots,n}$ and $\M$ a collection of partitions $m$ such that $S_{m}\subset S_{\mathfrak{m}}$. We obtain the following result.

\begin{prop}\label{pp}
Let $a,b>0$. Assume that
\begin{equation}\label{condppm}
|I|\ge (d+1)a^{2}\log^{2}n\ge d+1,\ \ \ \forall I\in \mathfrak{m}.
\end{equation}
If for some $K>1$, 
\[
\pen(m)\ge K\kappa^{2}\pa{\sigma^{2}+c{4\sqrt{2}(\sigma+c)(d+1)(b+2)\over a\kappa}}\pa{D_{m}+\Delta_{m}}\ \ \forall m\in\M
\] 
the estimator $\hat f_{\hat m}$ satisfies~\eref{inehisto} with 
\[
R=\kappa^{2}\pa{\sigma^{2}+c{4\sqrt{2}(\sigma+c)(d+1)(b+2)\over a\kappa}}\Sigma+4{(c+\sigma)^{2}(b+2)^{2}\over a^{2}n^{b}}.
\]
\end{prop}

\subsection{Families of trigonometric polynomials}\label{sect102}
We assume that $f$ has the same form as in Subsection~\ref{sect101}. Here, our aim is to estimate $F$ by a trigonometric polynomial of degree not larger than some $\overline D\ge 0$. 

Consider the (discrete) trigonometric system $\ac{\phi_{j}}_{j\ge 0}$  of vectors in $\R^{n}$ defined by
\begin{eqnarray*}
\phi_{0}&=&(1/\sqrt{n},\ldots,1/\sqrt{n})\\
\phi_{2j-1}&=&\sqrt{2\over n}\pa{\cos\pa{2\pi jx_{1}},\ldots,\cos\pa{2\pi jx_{1}}},\ \forall j\ge 1\\
\phi_{2j}&=&\sqrt{2\over n}\pa{\sin\pa{2\pi jx_{1}},\ldots,\sin\pa{2\pi jx_{1}}},\ \forall j\ge 1.
\end{eqnarray*}
Let $\M$ be a family of subsets of $\ac{0,\ldots,2\overline D}$. For $m\in\M$,  we define $S_{m}$ as the linear span of the $\phi_{j}$ with $j\in m$ (with the convention $S_{m}=\ac{0}$ when $m=\varnothing$). 

\begin{prop}\label{trigo}
Let $a,b>0$. Assume that $2\overline D+1\le \sqrt{n}/(a\log n)$.
If for some $K>1$,
\[
\pen(m)\ge K\kappa^{2}\pa{\sigma^{2}+{4c(c+\sigma)(b+2)\over a}}\pa{D_{m}+\Delta_{m}},\ \ \forall m\in\M
\]
then  $\hat f_{\hat m}$ satisfies~\eref{inehisto} with 
\[
R=\kappa^{2}\pa{\sigma^{2}+{4c(c+\sigma)(b+2)\over a}}\Sigma+{4(b+2)^{2}(c+\sigma)^{2}\over a^{2}(2\overline D +1)n^{b}}.
\]

\end{prop}

\section{Towards a more general result}\label{sec:unifions}
We consider the statistical framework presented in Section~\ref{sect:stats} and give a general result that allows to handle Propositions~\ref{histo}, ~
\ref{pp} and~\ref{trigo} simultaneously. It will rely on some geometric properties of the linear spaces $S_{m}$ that we describe below. 

\subsection{Some metric quantities}
Let $S$ be a linear subspace of $\R^{n}$. We associate to $S$ the following quantities 
\begin{equation}\label{defL}
\Lambda_{2}(S)=\max_{i=1,\ldots,n}|\Pi_{S}e_{i}|_{2}\ \ {\rm and}\ \ \Lambda_{\infty}(S)=\max_{i=1,\ldots,n}|\Pi_{S}e_{i}|_{1}.
\end{equation}
It is not difficult to see that these quantities can be interpreted in terms of norm connexions, more precisely
\[
\Lambda_{2}(S)=\sup_{t\in S\setminus\ac{0}}{\ab{t}_{\infty}\over \ab{t}_{2}}\ \ {\rm and}\ \ \Lambda_{\infty}(S)=\sup_{t\in \R^{n}\setminus\ac{0}}{\ab{\Pi_{S}t}_{\infty}\over \ab{t}_{\infty}}.
\]
Clearly, $\Lambda_{2}(S)\le 1$. Besides, since $\ab{x}_{1}\le \sqrt{n}\ab{x}_{2}$ for all $x\in\R^{n}$, $\Lambda_{\infty}(S)\le \sqrt{n}\Lambda_{2}(S)$. Nevertheless, these bounds can be rather rough and turn out to be much smaller for the linear spaces $S_{m}$ presented in Subsections~\ref{sect100},~\ref{sect101} and~\ref{sect102} (for the examples presented there, we refer to Subsections~\ref{Phisto},~\ref{Ppp} and~\ref{Ptrigo} respectively for more accurate upper bounds on those quantities). 

\subsection{The main result}\label{sect:main}
Let $\ac{S_{m},\ m\in\M}$ be family of linear spaces and $\ac{\Delta_{m},\ m\in\M}$ a family of nonnegative weights. We define $\S=\sum_{m\in\M}S_{m}$ and
\[
\overline{\Lambda}_{\infty}=\pa{\sup_{m,m'\in\M}\Lambda_{\infty}(S_{m}+S_{m'})}\vee 1.
\]

\begin{thm}\label{selmod}
Let $K>1$ and $z\ge 0$. Assume that for all $i=1,\ldots,n$, inequality~\eref{bernstein} holds.  Let $\pen$ be some penalty function satisfying 
\begin{equation}\label{pen}
\pen(m)\ge K\kappa^{2}\pa{\sigma^{2}+{2cu\over \kappa}}\pa{D_{m}+\Delta_{m}},\ \ \forall m\in\M
\end{equation}
where
\begin{equation}\label{defu}
u=(c+\sigma)\overline{\Lambda}_{\infty}\Lambda_{2}(\S)\log(n^{2}e^{z}).
\end{equation}
If one selects $\hat m$ among $\M$ as any minimizer of ${\rm crit}(.)$ defined by~\eref{crit} then 
\[
\E\cro{\ab{f-\hat f_{\hat m}}_{2}^{2}}\le C(K)\cro{\inf_{m\in\M}\pa{\E\cro{\ab{f-\hat f_{m}}_{2}^{2}}+\pen(m)}+R}
\]
where 
\begin{eqnarray}
C(K)&=& {K(K^{2}+K-1)\over (K-1)^{3}}\label{CK}
\end{eqnarray}
and $R=\kappa^{2}\pa{\sigma^{2}+2\kappa^{-1}cu}\Sigma+2u^{2}\overline{\Lambda}_{\infty}^{-2}e^{-z}$.
\end{thm}

When $c=0$ we derive the following corollary by letting $z$ grow towards infinity. 

\begin{cor}
Let $K>1$. Assume that the $\eps_{i}$ for $i=1,\ldots,n$ satisfy inequality~\eref{bernstein} with $c=0$. If one selects $\hat m$ among $\M$ as a minimizer of ${\rm crit}$ defined by~\eref{crit}
with $\pen$ satisfying
\[
\pen(m)\ge K\kappa^{2}\sigma^{2}\pa{D_{m}+\Delta_{m}},\ \ \forall m\in\M
\]
then 
\[
\E\cro{\ab{f-\hat f_{\hat m}}_{2}^{2}}\le {K(K^{2}+K-1)\over (K-1)^{3}}\inf_{m\in\M}\pa{\E\cro{\ab{f-\hat f_{m}}_{2}^{2}}+\pen(m)}+R
\]
where $R=K^{3}(K-1)^{-2}\kappa^{2}\sigma^{2}\Sigma$.
\end{cor}

\section{Proofs}\label{Proof}
We start with the following result generalizing
Theorem~\ref{norm} when $d$ and $\delta$ are not induced by norms. We assume that $T$ is finite and take numbers $v$ and $b$ such that 
\begin{equation}\label{contraintes}
\sup_{s\in T}d(s,t_{0})\le v,\ \ \ \sup_{s\in T}c\delta(s,t_{0})\le b.
\end{equation}
We consider now a family of finite partitions $\pa{\A_{k}}_{k\ge 0}$ of $T$, such that $\A_{0}=\ac{T}$ and for $k\ge 1$ and $A\in \A_{k}$
\[
d(s,t)\le 2^{-k}v\ \ {\rm and}\ \  c\delta(s,t)\le 2^{-k} b,\ \ \forall s,t\in A.
\]
Besides, we assume $\A_{k}\subset \A_{k-1}$ for all $k\ge 1$, which means that all elements $A\in \A_{k}$ are subsets of an element of $\A_{k-1}$. Finally, we define for $k\ge 0$
\[
N_{k}=|\A_{k+1}||\A_{k}|.
\]

\begin{thm}\label{chi2}
Let $T$ be some finite set. Under Assumption~\ref{momexp}, 
\begin{equation}\label{sva}
\P\pa{Z\ge H+2\sqrt{2v^{2}x}+ 2bx}\le e^{-x},\ \ \forall x>0
\end{equation}
where
\[
H=\sum_{k\ge 0}2^{-k}\pa{v\sqrt{2\log(2^{k+1}N_{k})}+b\log(2^{k+1}N_{k})}.
\]
Moreover, 
\begin{equation}\label{va}
\P\pa{\overline{Z}\ge H+2\sqrt{2v^{2}x}+ 2bx}\le 2e^{-x},\ \ \forall x>0.
\end{equation}
\end{thm}

The quantity $H$ can be related to the entropies of $T$ with respect to the distances $d$ and $c\delta$ (when $c\ne 0$) in the following way. We first recall that for a distance $e(.,.)$ on $T$ and $\varepsilon>0$, the entropy $H(T,e,\varepsilon)$ is defined as logarithm of the minimum number of balls of radius $\varepsilon$ with respect to $e$ which are necessary to cover $T$. For $\varepsilon>0$, let us set $H(T,\varepsilon)=\max\ac{H(T,d,\varepsilon v),H(T,c\delta,\varepsilon b)}$. Note that $H(T,\varepsilon)=0$ for $\varepsilon>1$ because of~\eref{contraintes}. For $\varepsilon<1$, one can bound $H(T,\varepsilon)$ from above as follows. For $k\ge 0$, each element $A$ of the partition $\A_{k+1}$ is both a subset of a ball of radius $2^{-(k+1)}v$ with respect to $d$ and of a ball of radius $2^{-(k+1)}b$ with respect $c\delta$. Since $|\A_{k+1}|\le N_{k}$, we obtain  for all $\varepsilon\in[2^{-(k+1)},2^{-k})$, $H(T,\varepsilon)\le \log N_{k}$ and by integrating with respect to $\varepsilon$ and summing over $k\ge 0$, we get
\[
\int_{0}^{1}\pa{\sqrt{2v^{2}H(T,\varepsilon)}+bH(T, \varepsilon)}d\varepsilon\le H.
\]

\subsection{Proof of Theorem~\ref{chi2}}
Note that we obtain~\eref{va} by using~\eref{sva} twice (once with $X_{t}$ and then with $-X_{t}$). Let us now prove~\eref{sva}. For each $k\ge 1$ and $A\in\A_{k}$, we choose some arbitrary element $t_{k}(A)$ in $A$. For each $t\in T$ and $k\ge 1$, there exists a unique $A\in \A_{k}$ such that $t\in A$ and we set $\pi_{k}(t)=t_{k}(A)$. When $k=0$, we set $\pi_{0}(t)=t_{0}$.

We consider the (finite) decomposition
\[
X_{t}-X_{t_{0}}=\sum_{k\ge 0}X_{\pi_{k+1}(t)}-X_{\pi_{k}(t)}
\]
and set for $k\ge 0$
\[
z_{k}=2^{-k}\pa{v\sqrt{2\pa{\log(2^{k+1}N_{k})+x}}\ +\ b\pa{\log(2^{k+1}N_{k})+x}}
\]
Since $\sum_{k\ge 0}z_{k}\le z=H+2v\sqrt{2x}+ 2bx$,
\begin{eqnarray*}
\P\pa{Z \ge z} &\le& \P\pa{\exists t,\ \exists k\ge 0,\ \ X_{\pi_{k+1}(t)}-X_{\pi_{k}(t)}\ge  z_{k}}\\
&\le& \sum_{k\ge 0}\sum_{(s,u)\in E_{k}}\P\pa{X_{u}-X_{s}\ge z_{k}}
\end{eqnarray*}
where 
\[
E_{k}=\ac{\pa{\pi_{k}(t),\pi_{k+1}(t)}|\ t\in T}.
\]
Since $\A_{k+1}\subset \A_{k}$, $\pi_{k}(t)$ and $\pi_{k+1}(t)$ belong to a same element of $\A_{k}$ and therefore $d(s,u)\le 2^{-k}v$ and $c\delta(s,u)\le 2^{-k}b$ for all pairs $(s,u)\in E_{k}$. 
Besides, under Assumption~\ref{momexp}, the random variable $X=X_{u}-X_{s}$ with $(s,u)\in E_{k}$ is centered and satisfies~\eref{exp} with $2^{-k}v$ and $2^{-k}b$ in place of $v$ and $c$. Hence, by using Bernstein's inequality~\eref{bern}, we get for all $(s,u)\in E_{k}$ and $k\ge 0$
\[
\P\pa{X_{u}-X_{s}\ge z_{k}}\le 2^{-(k+1)}N_{k}^{-1}e^{-x}\le 2^{-(k+1)}|E_{k}|^{-1}e^{-x}.
\]
Finally, we obtain inequality~\eref{sva} summing up this inequalities over $(s,u)\in E_{k}$ and $k\ge 0$. 

\subsection{Proof of Theorem~\ref{norm}}
We only prove~\eref{svanorm}, the argument for proving~\eref{vanorm} being the same as that for proving~\eref{va}. For $t\in S$ and $r>0$, we denote by $B_{2}(t,r)$ and $B_{\infty}(t,r)$ the balls centered at $t$ of radius $r$ associated to $\|\ \|_{2}$ and $\|\ \|_{\infty}$ respectively. In the sequel, we shall use the following result on the entropy of those balls. 

\begin{prop}\label{entropie}
Let $\|\ \|$ be an arbitrary norm on $S$ and $B(0,1)$ the corresponding unit ball. For each $\delta\in (0,1]$, the minimal number $\NN(\delta)$ of balls of radius $\delta$ (with respect to $\|\ \|$) which are necessary to cover $B(0,1)$ satisfies 
\[
\NN(\delta) \le \pa{1+2\delta^{-1}}^{D}.
\]
\end{prop}

This lemma can be found in Birg\'e~\citeyearpar{MR722129} (Lemma 4.5, p. 209) but we provide a proof below to keep this paper as self-contained as possible. 

\begin{proof}
With no loss of generality, we may assume that $S=\R^{D}$. Let $\delta\in (0,1]$. A subset $\T$ of $B(0,1)$ is called $\delta$-separated if for all $s,t\in\T$, $\|s-t\|>\delta$. If $\T$ is $\delta$-separated, the family of (open) balls centered at those $t\in\T$ with radius $\delta/2$ are all disjoint and included in the ball $B(0,1+\delta/2)$. By a volume argument (with respect to the Lebesgue measure on $\R^{D}$), we deduce that $\T$ is finite and satisfies  $|\T|\le (1+2\delta^{-1})^{D}$. Consider now a maximal $\delta$-separated set $\T$, that is  
\[
|\T|=\max_{\T'}|\T'|
\]
where $\T'$ runs among the family of all the $\delta$-separated subset of $B(0,1)$. By definition, for all $t\in B(0,1)\setminus \T$, $\T\cup\ac{t}$ is no longer a $\delta$-net and therefore that the family of balls $\ac{B(t,\delta),\ t\in\T}$ covers $B(0,1)$.  Consequently 
\[
\NN(\delta)\le |\T|\le (1+2\delta^{-1})^{D}.
\]
\end{proof}

Let us now turn to the proof of~\eref{svanorm}. Note that it is enough to prove that for some $u<H+2\sqrt{2v^{2}x}+ 2bx$ and all finite sets $T$ satisfying inequalities~\eref{debase} and~\eref{contraintes}  
\[
\P\pa{\sup_{t\in T}\pa{X_{t}-X_{t_{0}}} > u}\le e^{-x}.
\]
Indeed, for any sequence $\pa{T_{n}}_{n\ge 0}$ of finite subsets of $T$ increasing towards $T$, that is, satisfying $T_{n}\subset T_{n+1}$ for all $n\ge 0$ and $\bigcup_{n\ge 0}T_{n}=T$, the sets 
\[
\ac{\sup_{t\in T_{n}}\pa{X_{t}-X_{t_{0}}} > u}
\]
increases (for the inclusion) towards $\ac{Z>u}$. Therefore, 
\[
\P\pa{Z> u}=\lim_{n\to +\infty}\P\pa{\sup_{t\in T_{n}}\pa{X_{t}-X_{t_{0}}} > u}.
\]
Consequently, we shall assume hereafter that $T$ is finite. 

For $k\ge 0$ and $j\in\ac{2,\infty}$ define the sets $\A_{j,k}$ as follows. We first consider the case $j=2$. For $k=0$, $\A_{2,0}=\ac{T}$. By applying Proposition~\ref{entropie} with $\|\ \|=\|\ \|_{2}/v$ and $\delta=1/4$, we can cover $T\subset B_{2}(t_{0},v)$ with at most $9^{D}$ balls with radius $v/4$. From such a finite covering $\ac{B_{1},\ldots,B_{N}}$ with $N\le 9^{D}$, it is easy to derive a partition $\A_{2,1}$ of $T$ by at most $9^{D}$ sets of diameter not larger than $v/2$. Indeed, $\A_{2,1}$ can merely consist of the non-empty sets among the family 
\[
\ac{\pa{B_{k}\setminus\bigcup_{1\le \ell<k}B_{\ell}}\cap T,\ \ k=1,\ldots,N}
\]
(with the convention $\bigcup_{\varnothing}=\varnothing$). Then, for $k\ge 2$, proceed by induction using Proposition~\ref{entropie} repeatedly. Each element $A\in\A_{2,k-1}$ is a subset of a ball of radius $2^{-k}v$ and can be partitioned similarly as before into $5^{D}$ subsets of balls of radii $2^{-(k+1)}v$. By doing so, the partitions $\A_{2,k}$ with $k\ge 1$ satisfy $\A_{2,k}\subset \A_{2,k-1}$, $|\A_{2,k}|\le (1.8)^{D}\times 5^{kD}$ and for all $A\in\A_{2,k}$, 
\[
\sup_{s,t\in A}\|s-t\|_{2}\le 2^{-k}v.
\]
Let us now turn to the case $j=+\infty$. If $c>0$, define the partitions $\A_{\infty, k}$ in exactly the same way as we did for the $\A_{2, k}$. Similarly, the partitions $\A_{\infty,k}$ with $k\ge 1$ satisfy $\A_{\infty,k}\subset \A_{\infty,k-1}$, $|\A_{\infty,k}|\le (1.8)^{D}\times 5^{kD}$ and for all $A\in\A_{\infty,k}$, 
\[
\sup_{s,t\in A}c\|s-t\|_{\infty}\le 2^{-k}b.
\]
When $c=0$, we simply take $\A_{\infty, k}=\ac{T}$ for all $k\ge 0$ and note that the properties above are fulfilled as well.

Finally, define the partition $\A_{k}$ for $k\ge 0$ as that generated by $\A_{2,k}$ and $\A_{\infty,k}$, that is 
\[
\A_{k}=\ac{A_{2}\cap A_{\infty}|\ A_{2}\in\A_{2,k},\ A_{\infty}\in\A_{\infty,k}}.
\]
Clearly, $\A_{k+1}\subset \A_{k}$. Besides, $|\A_{0}|=1$ and for $k\ge 1$, 
\[
|\A_{k}|\le |\A_{2,k}||\A_{\infty,k}|\le (1.8)^{2D}\times 5^{2kD}.
\]
The set $T$ being finite, we can apply Theorem~\ref{chi2}. Actually, our construction of the $\A_{k}$ allows us to slightly gain in the constants. Going back to the proof of Theorem~\ref{chi2}, we note that 
\[
|E_{k}|=|\ac{\pa{\pi_{k}(t),\pi_{k+1}(t)}|\ t\in T}|\le |\A_{k+1}|\le 9^{2D}\times 5^{2kD}
\]
since the element $\pi_{k+1}(t)$ determines $\pi_{k}(t)$ in a unique way. This means that one can take $N_{k}=9^{2D}\times 5^{2kD}$ in the proof of Theorem~\ref{chi2}. By taking the notations of Theorem~\ref{chi2}, we have,  
\begin{eqnarray*}
H&\le &\sum_{k\ge 0}2^{-k}\cro{v\sqrt{2\log(2^{k+1}\times 9^{2D}\times 5^{2kD})}+b\log\pa{2^{k+1}\times 9^{2D}\times 5^{2kD}}}\\
&<& 14\sqrt{Dv^{2}}+18Db
\end{eqnarray*}
and using the concavity of $x\mapsto \sqrt{x}$, we get 
\begin{eqnarray*}
H+2\sqrt{2v^{2}x}+ 2bx &\le & 14\sqrt{Dv^{2}}+2\sqrt{2v^{2}x}+18b(D+x)\\
&\le& 18\pa{\sqrt{v^{2}\pa{D+x}}+b(D+x)}.
\end{eqnarray*}
which leads to the result.

\subsection{Control of $\chi^{2}$-type random variables}\label{sect-chi}
We have the following result. 

\begin{thm}\label{chi}
Let $S$ be some linear subspace of $\R^{n}$ with dimension $D$. If the coordinates of $\eps$ are independent and satisfy~\eref{bernstein}, for all $x,u>0$, 
\begin{equation}\label{c1}
\P\cro{|\Pi_{S}\eps|_{2}^{2}\ge \kappa^{2}\pa{\sigma^{2}+{2cu\over \kappa}}\pa{D+x},\ |\Pi_{S}\eps|_{\infty}\le u}\le e^{-x}
\end{equation}
with $\kappa=18$ and 
\begin{equation}\label{c2}
\P\pa{\ab{\Pi_{S}\eps}_{\infty}\ge x}\le 
2n\exp\cro{-{x^{2}\over 2\Lambda_{2}^{2}(S)\pa{\sigma^{2}+cx}}}
\end{equation}
where $\Lambda_{2}(S)$ is defined by~\eref{defL}.
\end{thm}

\begin{proof}
Let us set $\chi=|\Pi_{S}\eps|_{2}$. 
For $t\in S$, let $X_{t}=\<\eps,t\>$ and $t_{0}=0$. It follows from the independence of the $\eps_{i}$ and inequality~\eref{bernstein} that~\eref{debase} holds with $d(t,s)=\sigma|t-s|_{2}$ and $\delta(t,s)=|t-s|_{\infty}$, for all $s,t\in S$.  
The random variable $\chi$ equals the supremum of the $X_{t}$ when $t$ runs among the unit ball of $S$.  Besides, the supremum is achieved for $\hat t=\Pi_{S}\eps/\chi$ and thus, on the event $\ac{\chi\ge z,\ |\Pi_{S}\eps|_{\infty}\le u}$ 
\[
\chi=\sup_{t\in T}X_{t}\ \ {\rm with} \ T=\ac{t\in S,\ |t|_{2}\le 1,\ |t|_{\infty}\le  uz^{-1}}
\]
leading to the bound
\begin{eqnarray*}
\P\pa{\chi\ge z,\ |\Pi_{S}\eps|_{\infty}\le u}\le \P\pa{\sup_{t\in T}X_{t}\ge z}.
\end{eqnarray*}
We take $z=\kappa\sqrt{(\sigma^{2}+2cu\kappa^{-1})(D+x)}$ and (using the concavity of $x\mapsto \sqrt{x}$) note that 
\[
z\ge \kappa\pa{\sqrt{\sigma^{2}(D+x)}+cuz^{-1}(D+x)}.
\]
Then, by applying Theorem~\ref{norm} with $v=\sigma$, $b=cu/z$, we obtain~\eref{c1}.

Let us now turn to~\eref{c2}.
Under~\eref{bernstein}, we can apply Bernstein's inequality~\eref{bern} to $X=\<\eps,t\>$ and $X=\<-\eps,t\>$ with $t\in S$, $v^{2}=\sigma^{2}|t|_{2}^{2}$ and $c|t|_{\infty}$ in place of $c$ and get for all $t\in S$ and $x>0$
\begin{equation}\label{eq1}
\P\pa{|\<\eps,t\>|\ge x}\le 2\exp\cro{-{x^{2}\over 2\pa{\sigma^{2}|t|_{2}^{2}+c|t|_{\infty}x}}}.
\end{equation}
Let us take $t=\Pi_{S}e_{i}$ with $i\in\ac{1,\ldots,n}$. Since $|t|_{2}\le \Lambda_{2}(S)$ and 
\[
|t|_{\infty}=\max_{i,i'=1,\ldots,n}\ab{\<\Pi_{S}e_{i},e_{i'}\>}=\max_{i,i'=1,\ldots,n}\ab{\<\Pi_{S}e_{i},\Pi_{S}e_{i'}\>}\le \Lambda_{2}^{2}(S),
\]
for all $i\in\ac{1,\ldots,n}$
\begin{eqnarray*}
\P\pa{|\<\Pi_{S}\eps,e_{i}\>|\ge x}&\le& 2\exp\cro{-{x^{2}\over 2\Lambda_{2}^{2}(S)\pa{\sigma^{2}+cx}}}.
\end{eqnarray*}
We obtain~\eref{c2} by summing up these probabilities for $i=1,\ldots,n$.
\end{proof}

\subsection{Proof of Theorem~\ref{selmod}}
Let us fix some $m\in\M$. It follows from simple algebra and the inequality ${\rm crit}(\hat m)\le {\rm crit}(m)$ that 
\[
\ab{f-\hat f_{\hat m}}_{2}^{2}\le \ab{f-\hat f_{m}}_{2}^{2}+2\<\eps,\hat f_{\hat m}-\hat f_{m}\>+\pen(m)-\pen(\hat m).
\]
Using the elementary inequality $2ab\le a^{2}+b^{2}$  for all $a,b\in\R$, we have for  $K>1$,
\begin{eqnarray*}
2\<\eps,\hat f_{\hat m}-\hat f_{m}\>
&\le& 2\ab{\hat f_{\hat m}-\hat f_{m}}_{2}\ab{\Pi_{S_{m}+S_{\hat m}}\eps}_{2}\\
&\le& K^{-1}\cro{\pa{1+{K-1\over K}}\ab{\hat f_{\hat m}-f}_{2}^{2}+\pa{1+{K\over K-1}}\ab{f-\hat f_{m}}_{2}^{2}}\\
&& \ \ \ +\ \ K\ab{\Pi_{S_{m}+S_{\hat m}}\eps}_{2}^{2},
\end{eqnarray*}
and we derive
\begin{eqnarray*}
{(K-1)^{2}\over K^{2}}\ab{f-\hat f_{\hat m}}_{2}^{2}&\le& {K^{2}+K-1\over K(K-1)}\ab{f-\hat f_{m}}_{2}^{2}+K\ab{\Pi_{S_{m}+S_{\hat m}}\eps}_{2}^{2}-\pa{\pen(\hat m)-\pen(m)}\\
&\le& {K^{2}+K-1\over K(K-1)}\ab{f-\hat f_{m}}_{2}^{2}+\pen(m)\\
&& \ \ + K\ab{\Pi_{S_{m}+S_{\hat m}}\eps}_{2}^{2}-\pa{\pen(\hat m)+\pen(m)}.\\
\end{eqnarray*}
Setting 
\begin{eqnarray*}
A_{1}(\hat m)&=& K\kappa^{2}\pa{\sigma^{2}+{2cu\over \kappa}}\pa{{\ab{\Pi_{S_{m}+S_{\hat m}}\eps}_{2}^{2}\over \kappa^{2}\pa{\sigma^{2}+{2cu\over \kappa}}}-D_{\hat m}-D_{m}-\Delta_{\hat m}-\Delta_{m}}_{+}\1\ac{\ab{\Pi_{S_{m}+S_{\hat m}}\eps}_{\infty}\le u}\\
A_{2}(\hat m)&=& K\ab{\Pi_{S_{m}+S_{\hat m}}\eps}_{2}^{2}\1\ac{\ab{\Pi_{S_{m}+S_{\hat m}}\eps}_{\infty}\ge u}
\end{eqnarray*}
and using~\eref{pen}, we deduce that
\[
{(K-1)^{2}\over K^{2}}\ab{f-\hat f_{\hat m}}_{2}^{2}\le {K^{2}+K-1\over K(K-1)}\ab{f-\hat f_{m}}_{2}^{2}+\pen(m)+A_{1}(\hat m)+A_{2}(\hat m),
\]
and by taking the expectation on both side we get 
\[
{(K-1)^{2}\over K^{2}}\E\cro{\ab{f-\hat f_{\hat m}}_{2}^{2}}\le {K^{2}+K-1\over K(K-1)}\E\cro{\ab{f-\hat f_{m}}_{2}^{2}}+\pen(m)+\E\cro{A_{1}(\hat m)}+\E\cro{A_{2}(\hat m)}.
\]
The index $m$ being arbitrary, it remains to bound  $E_{1}=\E\cro{A_{1}(\hat m)}$ and $E_{2}=\E\cro{A_{2}(\hat m)}$ from above. 

Let $m'$ be some deterministic index in $\M$. By using Theorem~\ref{chi} with $S=S_{m}+S_{m'}$ the dimension of which is not larger than $D_{m}+D_{m'}$ and integrating~\eref{c1} with respect to $x$ we get
\[
\E\cro{A(m')}\le K\kappa^{2}\pa{\sigma^{2}+{2cu\over \kappa}}e^{-\Delta_{m}-\Delta_{m'}}
\] 
and thus 
\[
E_{1}\le \sum_{m'\in\M}\E\cro{A(m')}\le  K\kappa^{2}\pa{\sigma^{2}+{2cu\over \kappa}}\Sigma.
\]

Let us now turn to $\E\cro{A_{2}(\hat m)}$. By using that $S_{\hat m}+S_{m}\subset \S$, $\ab{\Pi_{S_{\hat m}+S_{m}}\xi}_{2}^{2}\le \ab{\Pi_{\S}\xi}_{2}^{2}\le n\ab{\Pi_{\S}\xi}_{\infty}^{2}$. Besides, it follows from the definition of $\overline \Lambda_{\infty}$ that 
\[
\ab{\Pi_{S_{\hat m}+S_{m}}\xi}_{\infty}= \ab{\Pi_{S_{\hat m}+S_{m}}\Pi_{\S}\xi}_{\infty}\le \overline \Lambda_{\infty}\ab{\Pi_{\S}\xi}_{\infty}.
\]
and therefore, setting $x_{0}=\overline \Lambda_{\infty}^{-1}u$
\begin{eqnarray*}
E_{2}&\le& Kn\E\cro{\ab{\Pi_{\S}\xi}_{\infty}^{2}\1\ac{\ab{\Pi_{\S}\xi}_{\infty}\ge x_{0}}}.
\end{eqnarray*}
We shall now use the following lemma the proof of which is deferred to the end of the section.
\begin{lem}\label{lem1}
Let $X$ be some nonnegative random variable satisfying for all $x>0$, 
\begin{equation}\label{bernX}
\P\pa{X\ge x}\le a\exp\cro{-\phi(x)}\ \ \ {\rm with}\ \ \ \phi(x)= {x^{2}\over 2\pa{\alpha+\beta x}}\ \ \ 
\end{equation}
where $a, \alpha>0$ and $\beta\ge 0$. For $x_{0}>0$ such that $\phi(x_{0})\ge 1$,
\[
\E\cro{X^{p}\1\ac{X\ge x_{0}}}\le ax_{0}^{p}e^{-\phi(x_{0})}\pa{1+{ep!\over \phi(x_{0})}},\ \ \ \forall p\ge 1.
\]
\end{lem}

We apply the lemma with $p=2$ and $X=\ab{\Pi_{\S}\xi}_{\infty}$ for which we know from~\eref{c2} that~\eref{bernX} holds with $a=2n$, $\alpha=\Lambda_{2}^{2}(S)\sigma^{2}$ and $\beta=\Lambda_{2}^{2}(S)c$. Besides, it follows from the definition of $x_{0}$ and the fact that $n\ge 2$ that
\[
\phi(x_{0})={x_{0}^{2}\over 2\Lambda_{2}^{2}(S)\pa{\sigma^{2}+cx_{0}}}\ge \log\pa{n^{2}e^{z}}\ge 1. 
\]
The assumptions of Lemma~\ref{lem1} being checked, we deduce that $E_{2}\le 2Kx_{0}^{2}e^{-z}$ and conclude the proof putting these upper bounds on $E_{1}$ and $E_{2}$ together. 

Let us now turn to the proof of the lemma.

\begin{proof}[Proof of Lemma~\ref{lem1}]
Since  
\[
\E\cro{X^{p}\1\ac{X\ge x_{0}}}\le x_{0}^{p}\P\pa{X\ge x_{0}}+\int_{x_{0}}^{+\infty}px^{p-1}\P\pa{X\ge x}dx,
\]
it remains to bound from above the integral. Let us set 
\[
I_{p}=\int_{x_{0}}^{+\infty}px^{p-1}e^{-\phi(x)}dx.
\]
Note that $\phi'$ is increasing and by integrating by parts we have  
\begin{eqnarray*}
I_{p}&=&\int_{x_{0}}^{+\infty}{px^{p-1}\over \phi'(x)}\phi'(x)e^{-\phi(x)}\le {p\over \phi'(x_{0})}\cro{x_{0}^{p-1}e^{-\phi(x_{0})}+(p-1)I_{p-1}}.
\end{eqnarray*}
By induction over $p$ and using that $x_{0}\phi'(x_{0})\ge \phi(x_{0})\ge 1$ we get
\begin{eqnarray*}
I_{p}&\le& p!x_{0}^{p}e^{-\phi(x_{0})}\sum_{k=0}^{p-1}{\pa{x_{0}\phi'(x_{0})}^{-(k+1)}\over (p-k-1)!}\le {ep!x_{0}^{p}e^{-\phi(x_{0})}\over \phi(x_{0})}.
\end{eqnarray*}
\end{proof}

\subsection{An intermediate result}
The following proposition allows to bound $\Lambda_{2}(S)$ and $\Lambda_{\infty}(S)$ under suitable assumptions on an orthonormal basis of $S$. 
\begin{prop}\label{control-lambda}
Let $P$ be some partition of $\ac{1,\ldots,n}$, $J$ some nonempty index set and 
\[
\ac{\phi_{j,I},\ (j,I)\in J\times P}
\]
an orthonormal system such that for some $\Phi>0$ and all $I\in P$
\[
\sup_{j\in J}\ab{\phi_{j,I}}_{\infty}\le {\Phi\over \sqrt{|I|}}\ \ {\rm and}\ \ \<\phi_{j,I},e_{i}\>=0\  \forall i\not \in I.
\]
If $S$ is the linear span of the $\phi_{j,I}$ with $(j,I)\in J\times P$, 
\[
\Lambda_{2}^{2}(S)\le \pa{{|J|\Phi^{2}\over \min_{I\in P}|I|}}\wedge 1\ \ {\rm and}\ \ \Lambda_{\infty}(S)\le \pa{|J|\Phi^{2}}\wedge \pa{\sqrt{n}\Lambda_{2}(S)}.
\]
\end{prop}

\begin{proof}[Proof of Proposition~\ref{control-lambda}]
We have already seen that $\Lambda_{2}(S)\le 1$ and $\Lambda_{\infty}(S)\le \sqrt{n}\Lambda_{2}(S)$, so it remains to show that
\[
\Lambda_{2}^{2}(S)\le {|J|\Phi^{2}\over \min_{I\in P}|I|}\ \ {\rm and}\ \ \Lambda_{\infty}(S)\le |J|\Phi^{2}.
\]
Let $i=1,\ldots,n$. There exists some unique $I\in P$ such that $i\in I$ and since $\<\phi_{j,I'},e_{i}\>=0$ for all $I'\neq I$, $\Pi_{S}e_{i}=\sum_{j\in J}\<e_{i},\phi_{j,I}\>\phi_{j,I}$. Consequently, 
\[
\ab{\Pi_{S}e_{i}}_{2}^{2}=\sum_{j\in J}\<e_{i},\phi_{j,I}\>^{2}\le {|J|\Phi^{2}\over |I|}\le {|J|\Phi^{2}\over \min_{I\in P}|I|}
\]
and 
\begin{eqnarray*}
\ab{\Pi_{S}e_{i}}_{1}&=&\sum_{i'\in I}\ab{\sum_{j\in J}\<e_{i},\phi_{j,I}\>\<e_{i'},\phi_{j,I}\>}\le |I|{|J|\Phi^{2}\over |I|}\le |J|\Phi^{2}.
\end{eqnarray*}
We conclude since  $i$ is arbitrary. 
\end{proof}

\subsection{Proof of Proposition~\ref{histo}}\label{Phisto}
Let $m$ be some partition of $\ac{1,\ldots, n}$. By applying 
Proposition~\ref{control-lambda} with $J=\ac{1}$, $P=m$ and $\Phi=1$, we obtain 
\[
\Lambda_{2}^{2}(S_{m})\le {1 \over \min_{I\in m}|I|}\ \ {\rm and}\ \ \Lambda_{\infty}(S_{m})\le 1.
\]
In fact, one can check that these inequalities are equalities. Since for all $m\in\M$, $S_{m}\subset S_{\mathfrak{m}}$, we deduce that under~\eref{condhisto}
\[
\Lambda_{2}^{2}(\S)\le \Lambda_{2}^{2}(S_{\mathfrak{m}})\le {1\over a^{2}\log^{2}n}
\]
For two partitions $m,m'$ of $\ac{1,\ldots, n}$, define 
\begin{equation}\label{suppart}
m\vee m'=\ac{I\cap I'|\  I\in m,\ I'\in m'}.
\end{equation}
Since the elements of $m,m'$ for $m,m'\in\M$ consist of consecutive integers $S_{m\vee m'}=S_{m}+S_{m'}$ and therefore
\[
\overline{\Lambda}_{\infty}=\sup_{m,m'\in\M}\Lambda_{\infty}(S_{m}+S_{m'})=\sup_{m,m'\in\M}\Lambda_{\infty}(S_{m\vee m'})=1.
\]
The result follows by applying Theorem~\ref{selmod} with $z=b\log n$.
 
\subsection{Proof of Proposition~\ref{pp}}\label{Ppp}
Let  $m$ be a partition of $\ac{1,\ldots,n}$ such that for all 
$I\in m$, $I$ consists of consecutive integers and $|I|> d$. As proved in Mason \& Handscom~\citeyearpar{MR1937591}, an orthonormal basis of $S_{m}$ is given by the vectors $\phi_{j,I}$ defined by 
\[
\<\phi_{0,I},e_{i}\>={1\over \sqrt{|I|}}\1_{I}(i)
\]
and for $j=1,\ldots,d$
\[
\<\phi_{j,I},e_{i}\>=\sqrt{2\over |I|}Q_{j}\pa{\cos\pa{{(i-\min I+1/2)\pi\over |I|}}}\1_{I}(i)
\]
where $Q_{j}$ is the Chebyshev polynomial of degree $j$ defined on $[-1,1]$ by the formula
\[
Q_{j}(x)=\cos(j\theta)\ \ {\rm if}\ \ x=\cos\theta.
\]
By applying Proposition~\ref{control-lambda} with $\Phi=\sqrt{2}$, $P=m$ and $J=\ac{0,\ldots,d}$ and get
\[
\Lambda_{2}^{2}(S_{m})\le {2(d+1)\over \min_{I\in m}|I|}\ \ {\rm and}\ \ \Lambda_{\infty}(S_{m})\le 2(d+1).
\]
Since for those $m\in\M$, $S_{m}\subset S_{\mathfrak{m}}$, $\S=\sum_{m\in \M}S_{m}\subset S_{\mathfrak{m}}$ and therefore
\[
\Lambda_{2}^{2}(\S)\le \Lambda_{2}^{2}(S_{\mathfrak{m}})\le {1\over a^{2}\log^{2}n}.
\]
Moreover, since for the elements of $m$ and $m'$ for $m,m'\in\M$ consist of consecutive integers $S_{m}+S_{m'}=S_{m\vee m'}$ with $m\vee m'$ is defined by~\eref{suppart} and 
\[
\sup_{m,m'\in\M}\Lambda_{\infty}(S_{m}+S_{m'})=\sup_{m,m'\in\M}\Lambda_{\infty}(S_{m\vee m'})\le 2(d+1)
\]
which implies that $\overline{\Lambda}_{\infty}\le 2(d+1)$. It remains to apply Theorem~\ref{selmod} with $z=b\log n$. 

\subsection{Proof of Proposition~\ref{trigo}}\label{Ptrigo}
Let $\mathfrak{m}=\ac{0,\ldots,2\overline D}$. Under the 	assumption that $2\overline D+1\le \sqrt{n}/(a\log n)$,  for all $m\subset\mathfrak{m}$, the family of vectors $\ac{\phi_{j}}_{j\in m}$ is a orthonormal basis of $S_{m}$. By applying Proposition~\ref{control-lambda} with $P$ reduced to $\ac{\ac{1,\ldots,n}}$, $J=m$, $\Phi=\sqrt{2}$, we get 
\[
\Lambda_{2}^{2}(S_{m})\le {2|m|\over n}\ \ {\rm and}\ \ \Lambda_{\infty}(S_{m})\le \sqrt{n}\Lambda_{2}(S_{m})\le \sqrt{2|m|}.
\]
Since for all $m\in \M$, $S_{m}\subset S_{\mathfrak{m}}$,  $\S=\sum_{m\in\M}S_{m}\subset S_{\mathfrak{m}}$ and therefore
\[
\Lambda_{2}^{2}(\S)\le \Lambda_{2}^{2}(S_{\mathfrak{m}})\le {2(2\overline D+1)\over n}.
\]
Moreover, for all $m,m'\in\M$, $S_{m}+S_{m'}=S_{m\cup m'}$ with $m\cup m'\subset \mathfrak{m}$ and thus, 
\[
\Lambda_{\infty}(S_{m}+S_{m'})\le \sqrt{2(|m\cup m'|}\le\sqrt{2(2\overline D+1)}. 
\]
It remains to apply Theorem~\ref{selmod} with $z=b\log n$. 

\thanks{Acknowledgment: We thank Jonas Kahn for pointing out this counter-example in Subsection~\ref{jonas} and to Lucien Birg\'e for his useful comments and for making us aware of the book of Talagrand which has been the starting point of this paper.}

\begin{thebibliography}{}

\bibitem[Baraud, 2000]{MR1777129}
Baraud, Y. (2000).
\newblock Model selection for regression on a fixed design.
\newblock {\em Probab. Theory Related Fields}, 117(4):467--493.

\bibitem[Baraud et~al., 2001]{MR1845321}
Baraud, Y., Comte, F., and Viennet, G. (2001).
\newblock Model selection for (auto-)regression with dependent data.
\newblock {\em ESAIM Probab. Statist.}, 5:33--49 (electronic).

\bibitem[Barron et~al., 1999]{MR1679028}
Barron, A., Birg{\'e}, L., and Massart, P. (1999).
\newblock Risk bounds for model selection via penalization.
\newblock {\em Probab. Theory Related Fields}, 113(3):301--413.

\bibitem[Barron and Cover, 1991]{MR1111806}
Barron, A.~R. and Cover, T.~M. (1991).
\newblock Minimum complexity density estimation.
\newblock {\em IEEE Trans. Inform. Theory}, 37(4):1034--1054.

\bibitem[Birg{\'e}, 1983]{MR722129}
Birg{\'e}, L. (1983).
\newblock Approximation dans les espaces m\'etriques et th\'eorie de
  l'estimation.
\newblock {\em Z. Wahrsch. Verw. Gebiete}, 65(2):181--237.

\bibitem[Birg{\'e} and Massart, 2001]{MR1848946}
Birg{\'e}, L. and Massart, P. (2001).
\newblock Gaussian model selection.
\newblock {\em J. Eur. Math. Soc. (JEMS)}, 3(3):203--268.

\bibitem[Bousquet, 2002]{MR1890640}
Bousquet, O. (2002).
\newblock A {B}ennett concentration inequality and its application to suprema
  of empirical processes.
\newblock {\em C. R. Math. Acad. Sci. Paris}, 334(6):495--500.

\bibitem[Bousquet, 2003]{MR2073435}
Bousquet, O. (2003).
\newblock Concentration inequalities for sub-additive functions using the
  entropy method.
\newblock In {\em Stochastic inequalities and applications}, volume~56 of {\em
  Progr. Probab.}, pages 213--247. Birkh\"auser, Basel.

\bibitem[Klartag and Mendelson, 2005]{MR2149924}
Klartag, B. and Mendelson, S. (2005).
\newblock Empirical processes and random projections.
\newblock {\em J. Funct. Anal.}, 225(1):229--245.

\bibitem[Klein and Rio, 2005]{MR2135312}
Klein, T. and Rio, E. (2005).
\newblock Concentration around the mean for maxima of empirical processes.
\newblock {\em Ann. Probab.}, 33(3):1060--1077.

\bibitem[Ledoux, 1996]{MR1399224}
Ledoux, M. (1996).
\newblock On {T}alagrand's deviation inequalities for product measures.
\newblock {\em ESAIM Probab. Statist.}, 1:63--87 (electronic).

\bibitem[Mason and Handscomb, 2003]{MR1937591}
Mason, J.~C. and Handscomb, D.~C. (2003).
\newblock {\em Chebyshev polynomials}.
\newblock Chapman \& Hall/CRC, Boca Raton, FL.

\bibitem[Massart, 2000]{MR1782276}
Massart, P. (2000).
\newblock About the constants in {T}alagrand's concentration inequalities for
  empirical processes.
\newblock {\em Ann. Probab.}, 28(2):863--884.

\bibitem[Massart, 2007]{MR2319879}
Massart, P. (2007).
\newblock {\em Concentration inequalities and model selection}, volume 1896 of
  {\em Lecture Notes in Mathematics}.
\newblock Springer, Berlin.
\newblock Lectures from the 33rd Summer School on Probability Theory held in
  Saint-Flour, July 6--23, 2003. With a foreword by Jean Picard.

\bibitem[Mendelson, 2008]{MR2368981}
Mendelson, S. (2008).
\newblock On weakly bounded empirical processes.
\newblock {\em Math. Ann.}, 340(2):293--314.

\bibitem[Mendelson et~al., 2007]{MR2373017}
Mendelson, S., Pajor, A., and Tomczak-Jaegermann, N. (2007).
\newblock Reconstruction and subgaussian operators in asymptotic geometric
  analysis.
\newblock {\em Geom. Funct. Anal.}, 17(4):1248--1282.

\bibitem[Rio, 2002]{MR1955352}
Rio, E. (2002).
\newblock Une in\'egalit\'e de {B}ennett pour les maxima de processus
  empiriques.
\newblock {\em Ann. Inst. H. Poincar\'e Probab. Statist.}, 38(6):1053--1057.
\newblock En l'honneur de J. Bretagnolle, D. Dacunha-Castelle, I. Ibragimov.

\bibitem[Sauv{\'e}, 2008]{MSauve}
Sauv{\'e}, M. (2008).
\newblock Histogram selection in non gaussian regression.
\newblock {\em ESAIM Probab. Statist.}, to appear.

\bibitem[Sudakov and Cirel'son, 1974]{MR0365680}
Sudakov, V.~N. and Cirel'son, B.~S. (1974).
\newblock Extremal properties of half-spaces for spherically invariant
  measures.
\newblock {\em Zap. Nau\v cn. Sem. Leningrad. Otdel. Mat. Inst. Steklov.
  (LOMI)}, 41:14--24, 165.
\newblock Problems in the theory of probability distributions, II.

\bibitem[Talagrand, 1995]{MR1361756}
Talagrand, M. (1995).
\newblock Concentration of measure and isoperimetric inequalities in product
  spaces.
\newblock {\em Inst. Hautes \'Etudes Sci. Publ. Math.}, (81):73--205.

\bibitem[Talagrand, 2005]{MR2133757}
Talagrand, M. (2005).
\newblock {\em The generic chaining}.
\newblock Springer Monographs in Mathematics. Springer-Verlag, Berlin.
\newblock Upper and lower bounds of stochastic processes.

\bibitem[van~de Geer, 1990]{Geer90}
van~de Geer, S. (1990).
\newblock Estimating a regression function.
\newblock {\em Ann. Statist.}, 18:907--924.

\end{thebibliography}
\bibliographystyle{apalike}

\end{document}